\documentclass[11pt]{article}
\usepackage[margin=1in]{geometry}
\usepackage[T1]{fontenc}
\usepackage[utf8]{inputenc}
\usepackage{lmodern}
\usepackage{microtype}
\usepackage{amsmath,amssymb,amsthm,mathtools}
\usepackage{enumitem}
\usepackage{graphicx}
\usepackage{booktabs,longtable,array}
\usepackage{float}
\usepackage{hyperref}
\usepackage{mathrsfs}
\usepackage{setspace}
\usepackage{xcolor}
\hypersetup{colorlinks=true,linkcolor=blue,citecolor=blue,urlcolor=blue}
\allowdisplaybreaks

\newtheorem{theorem}{Theorem}[section]
\newtheorem{proposition}[theorem]{Proposition}
\newtheorem{lemma}[theorem]{Lemma}
\newtheorem{corollary}[theorem]{Corollary}
\theoremstyle{definition}
\newtheorem{definition}[theorem]{Definition}
\newtheorem{problem}[theorem]{Problem}
\theoremstyle{remark}

\DeclareMathOperator{\supp}{supp}

\DeclareMathOperator{\Spec}{Spec}

\title{The Degree Landscape of the Partition Graph: Maximal Degree, Extremal Vertices, and Spectra}
\author{Fedor B. Lyudogovskiy}
\date{}

\begin{document}
\maketitle

\begin{abstract}
We study the degree landscape of the partition graph $G_n$, whose vertices are the integer partitions of $n$ and whose edges correspond to elementary transfers of one unit between parts, followed by reordering. Using the previously established local degree formula, we introduce the degree layers $D_d(n)$, the degree spectrum $\Spec_D(n)$, and the numerical invariants $\Delta_n$, $m_\Delta(n)$, and $s(n)$.

The main theorem provides an exact formula for the maximal degree. If
\[
\rho(n):=\max\{r:T_r\le n\},\qquad T_r=\frac{r(r+1)}{2},
\]
and
\[
\nu:=n-T_{\rho(n)},
\]
then
\[
\Delta_n=\rho(n)\bigl(\rho(n)-1\bigr)+\beta_{\rho(n)}(\nu),
\]
where $\beta_r$ is an explicit budget function governed by a square--pronic threshold rule. We also prove that every maximal-degree vertex lies on the maximal-support stratum, and we obtain exact extremal classifications at the levels $n=T_t$, $n=T_t+1$, and $n=T_t+2$.

The paper also includes a finite computation on the range $1\le n\le 60$, recording extremal multiplicities, representative extremal shapes, spectrum sizes, selected degree histograms, and first data on contact between the extremal layer and the self-conjugate axis. This computational part is deliberately limited in scope. It is descriptive rather than exhaustive, and is included only as a first numerical profile of the degree landscape.
\end{abstract}

\noindent\textbf{Keywords.} partition graph; integer partitions; vertex degree; extremal partitions; degree spectrum; self-conjugate axis.

\noindent\textbf{MSC 2020.} 05A17, 05C75, 05C69, 05C12.

\section{Introduction}
\label{sec:introduction}

For each $n$, let $G_n$ denote the graph whose vertices are the integer partitions of $n$, with two partitions adjacent whenever one is obtained from the other by an elementary transfer of one unit between parts, followed by reordering. This graph has already been studied from several complementary points of view in the preceding papers of the series: the clique complex $K_n=Cl(G_n)$ was shown to have a particularly simple global homotopy type \cite{LyudHomotopy}, the local structure through a fixed vertex was described in terms of ordered support data and explicit local invariants \cite{LyudLocal}, the graph was then viewed as a growing discrete geometric object with a boundary framework, central region, and anisotropic large-scale organization \cite{LyudGrowing}, and more recently its axial and simplex-layer morphologies were formalized through the self-conjugate axis, the spine, and the simplex stratification \cite{LyudAxial,LyudSimplex}.

This paper focuses on one specific aspect: the global landscape of vertex degrees. The degree of a vertex is one of the most basic local invariants of a graph, but in the partition graph it is governed by a surprisingly rigid combinatorial rule. If
\[
\lambda=(L_1^{m_1},L_2^{m_2},\dots,L_r^{m_r}),\qquad L_1>\cdots>L_r>0,
\]
and if $g_i=L_i-L_{i+1}$ with $L_{r+1}=0$, then the local degree formula gives
\[
\deg(\lambda)=r(r-1)+\sum_{i=1}^r \mathbf 1_{m_i>1}+\sum_{i=1}^r \mathbf 1_{g_i>1}.
\]
Thus the degree is controlled by three discrete ingredients: the support size, the multiplicity pattern, and the gap pattern. In \cite{LyudLocal}, this formula serves as part of a local, neighborhood-level description; here we employ it globally as the foundation for studying the degree landscape of $G_n$.

One aim of the paper is to formalize the degree-theoretic stratification
\[
V(G_n)=\bigsqcup_d D_d(n),\qquad D_d(n):=\{\lambda\vdash n:\deg(\lambda)=d\},
\]
and to study the associated degree spectrum
\[
\Spec_D(n):=\{\deg(\lambda):\lambda\vdash n\}.
\]
This leads naturally to the basic extremal parameters
\[
\Delta_n:=\max_{\lambda\vdash n}\deg(\lambda),\qquad m_\Delta(n):=|\{\lambda:\deg(\lambda)=\Delta_n\}|,\qquad s(n):=|\Spec_D(n)|.
\]
While the degree formula itself is local, these quantities are global: they describe the outer profile, multiplicity, and numerical width of the degree landscape across the entire graph.

A second aim is structural: we show that the local degree formula admits a useful global reinterpretation through a triangular decomposition of the total mass of a partition. This converts the degree problem into a constrained budget problem over support-maximal staircase backgrounds. As a consequence, extremal degree turns out to be governed by a rigid support principle: if $T_t=t(t+1)/2\le n<T_{t+1}$, then every vertex of maximal degree has exactly $t$ distinct part sizes. In other words, maximal degree is forced onto the maximal-support stratum.

Our main result is an exact formula for the maximal degree. Write
\[
\rho(n):=\max\{r:T_r\le n\},\qquad \nu:=n-T_{\rho(n)}.
\]
Then
\[
0\le \nu\le \rho(n),
\]
and we prove that
\[
\Delta_n=\rho(n)\bigl(\rho(n)-1\bigr)+\beta_{\rho(n)}(\nu),
\]
where $\beta_r(\nu)$ is the explicit bonus-budget function determined by the multiset
\[
\{1,1,2,2,\dots,r,r\}.
\]
Equivalently, if $n=T_t+\nu$ with $0\le \nu\le t$, then
\[
\Delta_n=t(t-1)+\beta_t(\nu).
\]
The function $\beta_t(\nu)$ is governed by a square--pronic threshold rule, so the exact maximal-degree sequence is piecewise controlled on each triangular interval $T_t\le n<T_{t+1}$.

This exact formula has several immediate consequences. It recovers the staircase partition as the unique maximizer at the triangular values $n=T_t$, gives the first nontrivial overlevels $T_t+1$ and $T_t+2$ explicitly, and shows that the maximal-degree sequence is governed piecewise on each triangular interval $T_t\le n<T_{t+1}$, rather than by a single uniform rule in $n$. In particular, the extremal degree is not merely bounded or approximated: it is determined exactly on every triangular interval.

A third aim is computational. Once $\Delta_n$ is known exactly, the next natural questions concern the geometry of the top degree layer rather than its height. We therefore study the multiplicity and shape of the extremal set
\[
\operatorname{MaxDeg}(n):=\{\lambda\vdash n:\deg(\lambda)=\Delta_n\},
\]
the number $s(n)$ of distinct degree values, selected degree histograms
\[
H_n(d):=|D_d(n)|,
\]
and a first comparison between the extremal layer and the self-conjugate axis $Ax_n$. Throughout the paper we keep the distinction between theorem-level statements and numerical observations explicit.

The main contributions of the paper may be summarized as follows.
\begin{enumerate}[label=\arabic*.]
\item We introduce the degree layers $D_d(n)$, the degree spectrum $\Spec_D(n)$, and the basic numerical invariants $\Delta_n$, $m_\Delta(n)$, and $s(n)$.
\item We reinterpret the local degree formula globally through a triangular mass decomposition and a weighted bonus-budget constraint.
\item We prove that every maximal-degree vertex lies on the maximal-support stratum.
\item We determine the maximal degree $\Delta_n$ exactly for every $n$, via the square--pronic budget rule on each triangular interval.
\item We obtain the first exact extremal classifications at $n=T_t$, $T_t+1$, and $T_t+2$.
\item We provide a finite computational profile on the range $1\le n\le 60$, covering extremal multiplicity, representative extremal shapes, degree spectra, selected histograms, and first axial contact data.
\end{enumerate}

The paper is organized as follows. Section~\ref{sec:preliminaries} fixes notation and recalls the imported degree formula. Section~\ref{sec:max-degree} develops the triangular mass decomposition, proves the maximal-support principle, derives the exact formula for $\Delta_n$, and records its first consequences. Section~\ref{sec:computational} gives a compact computational profile of the degree landscape on the range $1\le n\le 60$, including extremal multiplicities, representative extremal shapes, degree spectra, selected histograms, and a first axial comparison. Section~\ref{sec:conclusion} collects conclusions and open problems.

\section{Preliminaries and degree language}
\label{sec:preliminaries}

We write $G_n$ for the partition graph of $n$: its vertices are the integer partitions of $n$, and two vertices are adjacent if one is obtained from the other by an elementary transfer of one unit between two parts, followed by reordering.

We also write
\[
K_n:=Cl(G_n)
\]
for the clique complex of $G_n$.

Throughout the paper, for a partition $\lambda\vdash n$ we use the block form
\[
\lambda=(L_1^{m_1},L_2^{m_2},\dots,L_r^{m_r}),\qquad L_1>L_2>\cdots>L_r>0,\qquad m_i\ge 1,
\]
where $L_1,\dots,L_r$ are the distinct part sizes of $\lambda$, and $m_i$ is the multiplicity of $L_i$. To avoid conflict with the triangular parameter $t$ used later for the interval
\[
T_t\le n<T_{t+1},
\]
we use the symbols $L_i$ and $r$ rather than the symbols $s_i$ and $t$ used in the local paper.

We call
\[
r=r(\lambda):=|\supp(\lambda)|
\]
the support size of $\lambda$. We also set
\[
L_{r+1}:=0,\qquad g_i:=L_i-L_{i+1}\quad (1\le i\le r),
\]
and call the numbers $g_i$ the support gaps of $\lambda$.

Thus the local degree data of $\lambda$ consist of three ordered pieces: the support size $r$, the multiplicity pattern $(m_1,\dots,m_r)$, and the gap pattern $(g_1,\dots,g_r)$.

\subsection{Degree layers and degree spectrum}

\begin{definition}
For $d\ge 0$, the degree layer of level $d$ is
\[
D_d(n):=\{\lambda\vdash n:\deg_{G_n}(\lambda)=d\}.
\]
\end{definition}

\begin{definition}
The degree spectrum of $G_n$ is
\[
\Spec_D(n):=\{\deg_{G_n}(\lambda):\lambda\vdash n\}.
\]
\end{definition}

\begin{definition}
The maximal degree of $G_n$ is
\[
\Delta_n:=\max_{\lambda\vdash n}\deg_{G_n}(\lambda).
\]
\end{definition}

\begin{definition}
The extremal degree set is
\[
\operatorname{MaxDeg}(n):=\{\lambda\vdash n:\deg_{G_n}(\lambda)=\Delta_n\},
\]
and its cardinality is denoted by
\[
m_\Delta(n):=|\operatorname{MaxDeg}(n)|.
\]
\end{definition}

\begin{definition}
The number of distinct degree values in $G_n$ is
\[
s(n):=|\Spec_D(n)|.
\]
\end{definition}

We write
\[
Ax_n:=\{\lambda\vdash n: \lambda'=\lambda\}
\]
for the self-conjugate axis of $G_n$. When needed in open questions, we also use the previously established notation $Sp_n$ for the spine of the partition graph.

\subsection{Imported degree formula}

The key input is the explicit degree formula from the local theory.

\begin{proposition}[Degree formula]\label{prop:degree-formula}
Let
\[
\lambda=(L_1^{m_1},L_2^{m_2},\dots,L_r^{m_r})\vdash n,\qquad L_1>\cdots>L_r>0,
\]
and let
\[
g_i=L_i-L_{i+1}\quad (1\le i\le r),\qquad L_{r+1}=0.
\]
Then
\[
\deg(\lambda)=r(r+1)-\sum_{i=1}^r \mathbf 1_{m_i=1}-\sum_{i=1}^r \mathbf 1_{g_i=1}.
\]
Equivalently,
\[
\deg(\lambda)=r(r-1)+\sum_{i=1}^r \mathbf 1_{m_i>1}+\sum_{i=1}^r \mathbf 1_{g_i>1}.
\]
\end{proposition}

\begin{proof}
This is the local degree formula established in \cite{LyudLocal}, rewritten in the ordered-support notation used here.
\end{proof}

The second form will be more convenient for global degree geometry. It decomposes the degree into three parts: a quadratic support term $r(r-1)$, a multiplicity bonus, and a gap bonus.

\subsection{First consequences}

\begin{proposition}[Conjugation invariance of degree]\label{prop:conj-degree}
For every $n$ and every $\lambda\vdash n$,
\[
\deg(\lambda')=\deg(\lambda).
\]
Consequently, for every $d$,
\[
\lambda\in D_d(n)\quad\Longleftrightarrow\quad \lambda'\in D_d(n),
\]
so each degree layer $D_d(n)$, the spectrum $\Spec_D(n)$, and the extremal set $\operatorname{MaxDeg}(n)$ are conjugation-invariant.
\end{proposition}

\begin{proof}
If $\mu$ is obtained from $\lambda$ by an elementary transfer of one unit from one part to another, followed by reordering, then in Ferrers-diagram language $\mu$ is obtained from $\lambda$ by moving a single cell from one row to another. Under conjugation, rows and columns are exchanged, so the same operation becomes the transfer of one unit between two parts of the conjugate partition, again followed by reordering. Hence
\[
\lambda\sim\mu \quad\Longleftrightarrow\quad \lambda'\sim\mu',
\]
so conjugation defines an automorphism of $G_n$. In particular it preserves vertex degree, and the remaining claims follow immediately from the definitions.
\end{proof}

\begin{proposition}[Minimum degree]\label{prop:min-degree}
For every $n\ge 2$,
\[
\min_{\lambda\vdash n}\deg(\lambda)=1.
\]
Moreover, the only vertices of degree $1$ are the two antenna vertices
\[
(n),\qquad (1^n).
\]
\end{proposition}

\begin{proof}
Let
\[
\lambda=(L_1^{m_1},\dots,L_r^{m_r})\vdash n.
\]
If $r\ge 2$, then Proposition~\ref{prop:degree-formula} gives
\[
\deg(\lambda)\ge r(r-1)\ge 2.
\]
Hence a degree-$1$ vertex must have $r=1$, so $\lambda=(L_1^{m_1})$.

If $m_1=1$, then $\lambda=(n)$, and Proposition~\ref{prop:degree-formula} gives
\[
\deg(\lambda)=0+\mathbf 1_{g_1>1}=1,
\]
since $g_1=n>1$.

If $L_1=1$, then $\lambda=(1^n)$, and Proposition~\ref{prop:degree-formula} gives
\[
\deg(\lambda)=0+\mathbf 1_{m_1>1}=1
\]
for $n\ge 2$.

In all remaining one-block cases one has $L_1\ge 2$ and $m_1\ge 2$, hence both bonus terms are present and
\[
\deg(\lambda)=2.
\]
Therefore the only degree-$1$ vertices are $(n)$ and $(1^n)$.
\end{proof}

\begin{corollary}\label{cor:min-degree-class}
For every $n\ge 2$, one has
\[
1\in \Spec_D(n),
\]
and this value occurs with multiplicity exactly $2$.
\end{corollary}

At the opposite end of the degree spectrum, the degree formula suggests that the main extremal mechanism should be maximal support size. The next section begins the formal development of this principle.

\section{Exact extremal degree}
\label{sec:max-degree}

\subsection{A global mass decomposition behind the degree formula}

The degree formula of Proposition~\ref{prop:degree-formula} is local in appearance, but it admits a useful global reinterpretation. The point is that the support size $r$, the gaps $g_i$, and the multiplicities $m_i$ are not independent: they are constrained by the total mass $n$. The following identity makes this relation explicit.

Let
\[
T_r:=\frac{r(r+1)}{2}
\]
denote the $r$-th triangular number.

\begin{lemma}[Triangular decomposition of the mass]\label{lem:triangular-mass}
Let
\[
\lambda=(L_1^{m_1},L_2^{m_2},\dots,L_r^{m_r})\vdash n,\qquad L_1>\cdots>L_r>0,\qquad L_{r+1}=0,
\]
and let
\[
g_i:=L_i-L_{i+1}\qquad (1\le i\le r).
\]
Then
\[
n=T_r+\sum_{j=1}^r j(g_j-1)+\sum_{i=1}^r (m_i-1)L_i.
\]
\end{lemma}

\begin{proof}
Since
\[
n=\sum_{i=1}^r m_iL_i,
\]
it is enough to express $\sum_{i=1}^r L_i$ in terms of the gap data. Because
\[
L_i=(r+1-i)+\sum_{j=i}^r (g_j-1),
\]
we obtain
\[
\sum_{i=1}^r L_i=\sum_{i=1}^r (r+1-i)+\sum_{i=1}^r\sum_{j=i}^r (g_j-1).
\]
The first sum is $T_r$. Reversing the order in the double sum gives
\[
\sum_{i=1}^r\sum_{j=i}^r (g_j-1)=\sum_{j=1}^r j(g_j-1).
\]
Hence
\[
\sum_{i=1}^r L_i=T_r+\sum_{j=1}^r j(g_j-1).
\]
Now
\[
n=\sum_{i=1}^r m_iL_i=\sum_{i=1}^r L_i+\sum_{i=1}^r (m_i-1)L_i,
\]
and substituting the previous expression for $\sum_i L_i$ yields the claim.
\end{proof}

\begin{proposition}[Weighted bonus inequality]\label{prop:weighted-bonus}
Let
\[
\lambda=(L_1^{m_1},L_2^{m_2},\dots,L_r^{m_r})\vdash n,\qquad L_1>\cdots>L_r>0,
\]
and let
\[
g_i=L_i-L_{i+1}\qquad (1\le i\le r),\qquad L_{r+1}=0.
\]
Set
\[
s:=n-T_r,\qquad T_r=\frac{r(r+1)}{2}.
\]
Then
\[
\sum_{j=1}^r j\,\mathbf 1_{g_j>1}+\sum_{i=1}^r (r+1-i)\,\mathbf 1_{m_i>1}\le s.
\]
\end{proposition}

\begin{proof}
By Lemma~\ref{lem:triangular-mass},
\[
s=\sum_{j=1}^r j(g_j-1)+\sum_{i=1}^r (m_i-1)L_i.
\]
If $g_j>1$, then $g_j-1\ge 1$, hence
\[
j(g_j-1)\ge j\,\mathbf 1_{g_j>1}.
\]
Also, since $L_1>\cdots>L_r>0$, one has
\[
L_i\ge r+1-i\qquad (1\le i\le r).
\]
Therefore, if $m_i>1$, then $m_i-1\ge 1$, and hence
\[
(m_i-1)L_i\ge (r+1-i)\,\mathbf 1_{m_i>1}.
\]
Summing these inequalities and using the expression for $s$, we obtain the claim.
\end{proof}

\begin{definition}
For integers $r\ge 1$ and $s\ge 0$, let $\beta_r(s)$ denote the maximal cardinality of a submultiset of
\[
\{1,1,2,2,\dots,r,r\}
\]
whose total weight is at most $s$.

Equivalently, $\beta_r(s)$ is the maximal number of bonus terms that can be activated under the weighted budget $s$. Here one copy of the weight $j$ corresponds to the gap bonus $g_j>1$, while the other corresponds to the multiplicity bonus carried by the $(r+1-j)$-st support level, that is, to the condition $m_{r+1-j}>1$.
\end{definition}

Thus the multiplicity bonuses are indexed from the bottom upward, so that the bottom-most multiplicities carry the smallest weights in the budget and match the second copy of each weight $j$.

\begin{corollary}[Support-wise degree bound]\label{cor:support-wise-bound}
Let $\lambda\vdash n$ have support size $r$, and let $s=n-T_r$. Then
\[
\deg(\lambda)\le r(r-1)+\beta_r(s).
\]
In particular,
\[
\deg(\lambda)\le r(r+1).
\]
\end{corollary}

\begin{proof}
By Proposition~\ref{prop:degree-formula},
\[
\deg(\lambda)=r(r-1)+\sum_{j=1}^r \mathbf 1_{g_j>1}+\sum_{i=1}^r \mathbf 1_{m_i>1}.
\]
By Proposition~\ref{prop:weighted-bonus}, the activated bonus terms satisfy the weighted budget constraint
\[
\sum_{j=1}^r j\,\mathbf 1_{g_j>1}+\sum_{i=1}^r (r+1-i)\,\mathbf 1_{m_i>1}\le s.
\]
Hence the total number of active bonus terms is at most $\beta_r(s)$, proving the first inequality. The second follows from the trivial bound $\beta_r(s)\le 2r$.
\end{proof}

\begin{theorem}[Maximal-support principle for extremal degree]\label{thm:max-support-principle}
Let
\[
\rho(n):=\max\{r:T_r\le n\}=\left\lfloor\frac{\sqrt{8n+1}-1}{2}\right\rfloor.
\]
If $\lambda\in \operatorname{MaxDeg}(n)$, then
\[
|\supp(\lambda)|=\rho(n).
\]
\end{theorem}

\begin{proof}
Let $\rho=\rho(n)$. Since $T_\rho\le n$, the staircase partition
\[
\delta_\rho:=(\rho,\rho-1,\dots,1)
\]
is a partition of $T_\rho$, and by adding the remaining $n-T_\rho$ units to the largest part one obtains a partition of $n$ with support size $\rho$. Therefore
\[
\Delta_n\ge \rho(\rho-1),
\]
because the support term alone already contributes $\rho(\rho-1)$.

Now let $\lambda\vdash n$ have support size $r<\rho$.

If $r\le \rho-2$, then Corollary~\ref{cor:support-wise-bound} gives
\[
\deg(\lambda)\le r(r+1)\le (\rho-2)\rho<\rho(\rho-1)\le \Delta_n.
\]
Hence such a partition cannot be extremal.

Suppose next that $r=\rho-1$, and write
\[
s=n-T_r.
\]
Since $T_\rho\le n<T_{\rho+1}$, one has
\[
\rho\le s\le 2\rho.
\]
By Corollary~\ref{cor:support-wise-bound},
\[
\deg(\lambda)\le r(r-1)+\beta_r(s).
\]
Now the total weight of the full multiset
\[
\{1,1,2,2,\dots,r,r\}
\]
is
\[
2(1+\cdots+r)=r(r+1)=\rho(\rho-1).
\]
For $\rho\ge 4$, we have
\[
s\le 2\rho<\rho(\rho-1),
\]
so not all $2r$ bonus terms can be active. Hence
\[
\beta_r(s)\le 2r-1,
\]
and therefore
\[
\deg(\lambda)\le r(r-1)+(2r-1)=r(r+1)-1=\rho(\rho-1)-1.
\]
Thus $\deg(\lambda)<\Delta_n$, so $\lambda$ cannot be extremal.

It remains only to handle the cases $\rho\le 3$, which correspond to $n\le 9$. For $n=1,2$ the claim is immediate. For $n=3,4,5$, where $\rho=2$, the partitions with support size $1$ have maximal degree $1,2,1$, respectively, while $\Delta_n=2,3,4$. For $n=6,7,8,9$, where $\rho=3$, the partitions with support size $2$ have maximal degree $4,4,5,5$, respectively, while $\Delta_n=6,7,8,8$. Hence no partition with support size $\rho(n)-1$ is extremal in these small cases either. Therefore every extremal vertex has support size $\rho(n)$.
\end{proof}

\subsection{Explicit evaluation of the bonus budget function}

For $k\ge 0$, define
\[
w(k):=\begin{cases}0, & k=0,\\[1mm] q^2, & k=2q-1\ \text{for some } q\ge 1,\\[1mm] q(q+1), & k=2q\ \text{for some } q\ge 1.\end{cases}
\]
Thus $w(k)$ is the total weight of the $k$ smallest elements of the infinite multiset
\[
\{1,1,2,2,3,3,\dots\}.
\]

\begin{lemma}[Minimal weight for $k$ bonus terms]\label{lem:min-weight-k-bonus}
Let $1\le k\le 2r$. Among all $k$-element submultisets of
\[
\{1,1,2,2,\dots,r,r\},
\]
the minimal possible total weight is $w(k)$.
\end{lemma}

\begin{proof}
To minimize the total weight of a $k$-element submultiset, one must choose the $k$ smallest available elements. Since the ambient multiset contains exactly two copies of each positive integer from $1$ to $r$, the claim follows immediately.
\end{proof}

\begin{proposition}[Explicit formula for $\beta_r(s)$]\label{prop:beta-formula}
For every $r\ge 1$ and $s\ge 0$,
\[
\beta_r(s)=\max\{k\in\{0,1,\dots,2r\}: w(k)\le s\}.
\]
Equivalently, $\beta_r(s)$ is the unique integer $k\in[0,2r]$ satisfying
\[
w(k)\le s < w(k+1),
\]
with the convention $w(2r+1)=+\infty$.

In particular, if $0\le s<r(r+1)$, then $\beta_r(s)$ is given by the square--pronic rule: for $q=\lfloor \sqrt{s}\rfloor$,
\[
\beta_r(s)=\begin{cases}2q-1, & q^2\le s<q(q+1),\\[1mm] 2q, & q(q+1)\le s<(q+1)^2,\end{cases}
\]
provided $q<r$; and if $r^2\le s<r(r+1)$, then
\[
\beta_r(s)=2r-1.
\]
Finally, for $s\ge r(r+1)$,
\[
\beta_r(s)=2r.
\]
\end{proposition}

\begin{proof}
By definition, $\beta_r(s)$ is the maximum number of selected items under total weight at most $s$. By Lemma~\ref{lem:min-weight-k-bonus}, a $k$-element choice is feasible if and only if
\[
w(k)\le s.
\]
Hence the displayed formula follows. The explicit piecewise form is obtained from
\[
w(2q-1)=q^2,\qquad w(2q)=q(q+1).
\]
\end{proof}

\begin{corollary}[Explicit support-wise degree bound]\label{cor:explicit-support-wise-bound}
Let $\lambda\vdash n$ have support size $r$, and set
\[
s=n-T_r.
\]
Then
\[
\deg(\lambda)\le r(r-1)+\beta_r(s),
\]
where $\beta_r(s)$ is given explicitly by Proposition~\ref{prop:beta-formula}.
\end{corollary}

\subsection{First exact consequences and staircase perturbations}

For $t\ge 2$, let
\[
\delta_t:=(t,t-1,\dots,2,1)
\]
be the staircase partition of $T_t=t(t+1)/2$. We introduce three elementary perturbations of $\delta_t$:
\[
\delta_t^{\mathrm{top}}:=(t+1,t-1,t-2,\dots,2,1),
\]
\[
\delta_t^{\mathrm{bot}}:=(t,t-1,\dots,2,1,1),
\]
and
\[
\delta_t^{\mathrm{tb}}:=(t+1,t-1,t-2,\dots,2,1,1).
\]

\begin{proposition}[First staircase perturbations]\label{prop:first-staircase-perturbations}
For every $t\ge 2$,
\[
|\delta_t^{\mathrm{top}}|=|\delta_t^{\mathrm{bot}}|=T_t+1,\qquad |\delta_t^{\mathrm{tb}}|=T_t+2.
\]
Moreover,
\[
\deg(\delta_t^{\mathrm{top}})=\deg(\delta_t^{\mathrm{bot}})=t(t-1)+1,
\]
and
\[
\deg(\delta_t^{\mathrm{tb}})=t(t-1)+2.
\]
Finally,
\[
(\delta_t^{\mathrm{top}})'=\delta_t^{\mathrm{bot}},\qquad (\delta_t^{\mathrm{tb}})'=\delta_t^{\mathrm{tb}}.
\]
\end{proposition}

\begin{proof}
Immediate from Proposition~\ref{prop:degree-formula} and the explicit gap and multiplicity patterns.
\end{proof}

\begin{theorem}[Exact extremal theorem for triangular numbers]\label{thm:triangular-extremal}
If
\[
n=T_t=\frac{t(t+1)}{2},
\]
then
\[
\Delta_n=t(t-1),
\]
and the unique vertex of degree $\Delta_n$ is the staircase partition
\[
\delta_t=(t,t-1,\dots,1).
\]
\end{theorem}

\begin{proof}
By Theorem~\ref{thm:max-support-principle}, every extremal partition must have support size $t$. Since $T_t$ is the minimal possible mass of a partition with $t$ distinct positive part sizes, Lemma~\ref{lem:triangular-mass} forces all support gaps and multiplicities to be minimal, hence the only possibility is $\delta_t$. The degree formula then gives $\deg(\delta_t)=t(t-1)$.
\end{proof}

\begin{theorem}[Exact extremal classification at $T_t+1$]\label{thm:triangular-plus-one}
Let $t\ge 2$, and set
\[
n=T_t+1.
\]
Then
\[
\Delta_n=t(t-1)+1,
\]
and
\[
\operatorname{MaxDeg}(n)=\{\delta_t^{\mathrm{top}},\delta_t^{\mathrm{bot}}\}.
\]
In particular,
\[
m_\Delta(T_t+1)=2.
\]
\end{theorem}

\begin{proof}
By Corollary~\ref{cor:explicit-support-wise-bound} with $r=t$ and $s=1$,
\[
\Delta_n\le t(t-1)+\beta_t(1)=t(t-1)+1.
\]
Proposition~\ref{prop:first-staircase-perturbations} gives equality. To classify extremals, let $\lambda$ be extremal. By Theorem~\ref{thm:max-support-principle}, it has support size $t$, and by Lemma~\ref{lem:triangular-mass} its total surplus over the staircase background is exactly $1$. Since the degree is $t(t-1)+1$, there is exactly one active bonus term. Proposition~\ref{prop:weighted-bonus} then forces that active bonus to have weight $1$. The only weight-$1$ bonuses on the maximal-support stratum are the top-gap bonus and the bottom-multiplicity bonus, so $\lambda$ must be either $\delta_t^{\mathrm{top}}$ or $\delta_t^{\mathrm{bot}}$.
\end{proof}

\begin{theorem}[Exact extremal classification at $T_t+2$]\label{thm:triangular-plus-two}
Let $t\ge 2$, and set
\[
n=T_t+2.
\]
Then
\[
\Delta_n=t(t-1)+2,
\]
and the unique extremal vertex is
\[
\delta_t^{\mathrm{tb}}=(t+1,t-1,t-2,\dots,2,1,1).
\]
In particular,
\[
m_\Delta(T_t+2)=1.
\]
\end{theorem}

\begin{proof}
By Corollary~\ref{cor:explicit-support-wise-bound} with $r=t$ and $s=2$,
\[
\Delta_n\le t(t-1)+\beta_t(2)=t(t-1)+2.
\]
Proposition~\ref{prop:first-staircase-perturbations} shows that this value is attained. If $\lambda$ is extremal, then by Theorem~\ref{thm:max-support-principle} it has support size $t$, and by Lemma~\ref{lem:triangular-mass} the total surplus is $2$. Since the degree is $t(t-1)+2$, there are exactly two active bonus terms. Proposition~\ref{prop:weighted-bonus} therefore forces the sum of their weights to be at most $2$, and because both weights are positive integers, both active bonuses must have weight $1$. The only weight-$1$ bonuses on the maximal-support stratum are the top-gap bonus and the bottom-multiplicity bonus, so necessarily $\lambda=\delta_t^{\mathrm{tb}}$.
\end{proof}

\subsection{Mixed staircase perturbations and exact extremal degree on triangular intervals}

For integers $t\ge 1$, $a,b,c\ge 0$ with
\[
a+b\le t,
\]
and, in addition, with the extra condition that if $a=0$ then $c=0$, define the mixed staircase perturbation
\[
\Lambda_t(a,b;c)
\]
by prescribing its distinct part sizes and multiplicities as follows. For $1\le i\le t$, set
\[
L_i:=t+1-i+\max\{a+1-i,0\}+c\,\mathbf 1_{i=1},
\]
and
\[
m_i:=\begin{cases}1, & 1\le i\le t-b,\\[1mm] 2, & t-b+1\le i\le t.\end{cases}
\]
Then we define
\[
\Lambda_t(a,b;c):=(L_1^{m_1},L_2^{m_2},\dots,L_t^{m_t}).
\]

The condition $a+b\le t$ ensures that the support levels carrying the duplicated bottom parts lie strictly below the support levels modified by the top gap perturbation, while the extra condition $a=0 \Rightarrow c=0$ ensures that the top support remains strictly decreasing even when no top-gap perturbation is present.

\begin{proposition}[Size and degree of mixed staircase perturbations]\label{prop:mixed-perturbations}
Let $a,b,c$ satisfy the above assumptions. Then $\Lambda_t(a,b;c)$ has support size $t$, and
\[
|\Lambda_t(a,b;c)|=T_t+T_a+T_b+c,
\]
where
\[
T_m=\frac{m(m+1)}{2}.
\]
Moreover,
\[
\deg(\Lambda_t(a,b;c))=t(t-1)+a+b.
\]
\end{proposition}

\begin{proof}
By construction, the sequence $L_1>\cdots>L_t>0$ is strictly decreasing, so $\Lambda_t(a,b;c)$ has support size $t$. Indeed, if $a\ge 1$, then
\[
L_1-L_2=2+c,\qquad L_i-L_{i+1}=2\quad (2\le i\le a),\qquad L_i-L_{i+1}=1\quad (i>a),
\]
while if $a=0$, then necessarily $c=0$ and we recover the staircase support.

Summing the distinct part sizes gives
\[
\sum_{i=1}^t L_i=\sum_{i=1}^t (t+1-i)+\sum_{i=1}^a (a+1-i)+c=T_t+T_a+c.
\]
Next, since $a+b\le t$, the duplicated support levels are indexed by $i=t-b+1,\dots,t$, which lie below the modified top-gap region. Hence on these levels one has
\[
L_{t-b+1}=b,\quad L_{t-b+2}=b-1,\quad \dots,\quad L_t=1,
\]
so the extra contribution coming from the duplicated bottom parts is
\[
\sum_{i=t-b+1}^t L_i=1+2+\cdots+b=T_b.
\]
Therefore
\[
|\Lambda_t(a,b;c)|=\sum_{i=1}^t m_iL_i=\sum_{i=1}^t L_i+\sum_{i=t-b+1}^t L_i=T_t+T_a+c+T_b.
\]

Finally, the active gap bonuses are exactly the first $a$ ones, since
\[
g_i=L_i-L_{i+1}>1 \iff 1\le i\le a,
\]
and the active multiplicity bonuses are exactly the bottom $b$ ones, since
\[
m_i>1 \iff t-b+1\le i\le t.
\]
Applying Proposition~\ref{prop:degree-formula} yields
\[
\deg(\Lambda_t(a,b;c))=t(t-1)+a+b.
\]
This proves the claim.
\end{proof}

\begin{theorem}[Exact formula for the maximal degree on a triangular interval]\label{thm:exact-max-degree-triangular-interval}
Let
\[
n=T_t+\nu,\qquad 0\le \nu\le t.
\]
Then
\[
\Delta_n=t(t-1)+\beta_t(\nu).
\]

Equivalently, if $q=\lfloor \sqrt{\nu}\rfloor$, then
\[
\Delta_{T_t+\nu}=\begin{cases}t(t-1), & \nu=0,\\[1mm] t(t-1)+2q-1, & q^2\le \nu<q(q+1),\\[1mm] t(t-1)+2q, & q(q+1)\le \nu<(q+1)^2.\end{cases}
\]
\end{theorem}

\begin{proof}
The upper bound
\[
\Delta_n\le t(t-1)+\beta_t(\nu)
\]
follows from Theorem~\ref{thm:max-support-principle} and Corollary~\ref{cor:explicit-support-wise-bound}.

It remains to prove the reverse inequality by explicit construction. If $\nu=0$, this is Theorem~\ref{thm:triangular-extremal}. Assume therefore that $\nu\ge 1$, and let $q=\lfloor \sqrt{\nu}\rfloor$. In both constructions below, the parameter $c$ is nonnegative, and the extra condition $a=0 \Rightarrow c=0$ is automatic because $a=q$ and $q\ge 1$ when $\nu\ge 1$.

If
\[
q^2\le \nu<q(q+1),
\]
set
\[
a=q,\qquad b=q-1,\qquad c=\nu-q^2.
\]
Then
\[
T_a+T_b+c=T_q+T_{q-1}+\nu-q^2=\nu.
\]
Moreover,
\[
a+b=2q-1\le q^2\le \nu\le t,
\]
so Proposition~\ref{prop:mixed-perturbations} applies and gives a partition of size $T_t+\nu$ with degree
\[
t(t-1)+a+b=t(t-1)+2q-1=t(t-1)+\beta_t(\nu).
\]

If
\[
q(q+1)\le \nu<(q+1)^2,
\]
set
\[
a=q,\qquad b=q,\qquad c=\nu-q(q+1).
\]
Then
\[
T_a+T_b+c=T_q+T_q+\nu-q(q+1)=\nu,
\]
and
\[
a+b=2q\le q(q+1)\le \nu\le t.
\]
Again Proposition~\ref{prop:mixed-perturbations} applies and gives a partition of size $T_t+\nu$ with degree
\[
t(t-1)+a+b=t(t-1)+2q=t(t-1)+\beta_t(\nu).
\]

Thus the upper bound is attained in all cases, and hence
\[
\Delta_n=t(t-1)+\beta_t(\nu).
\]
This proves the theorem.
\end{proof}

\begin{corollary}\label{cor:exact-max-degree}
Let
\[
\rho(n)=\max\{r:T_r\le n\},\qquad \nu=n-T_{\rho(n)}.
\]
Then
\[
\Delta_n=\rho(n)(\rho(n)-1)+\beta_{\rho(n)}(\nu).
\]
\end{corollary}

\subsection{Consequences of the exact maximal-degree formula}

We now record several direct consequences of the exact formula
\[
\Delta_{T_t+\nu}=t(t-1)+\beta_t(\nu)
\]
on each triangular interval.

\begin{proposition}[Monotonicity pattern of the maximal degree sequence]\label{prop:monotonicity-pattern}
The sequence $(\Delta_n)_{n\ge 1}$ is nondecreasing.

Moreover:
\begin{enumerate}[label=(\roman*)]
\item if $T_t\le n<n+1<T_{t+1}$, then
\[
\Delta_{n+1}-\Delta_n\in\{0,1\};
\]
\item if $n=T_{t+1}-1$, then
\[
\Delta_{n+1}>\Delta_n.
\]
\end{enumerate}
\end{proposition}

\begin{proof}
Inside a fixed triangular interval, the claim follows from
\[
\Delta_{T_t+\nu}=t(t-1)+\beta_t(\nu)
\]
and the monotonicity of $\beta_t(\nu)$ in $\nu$.

At the boundary,
\[
\Delta_{T_{t+1}}=(t+1)t,\qquad \Delta_{T_{t+1}-1}=t(t-1)+\beta_t(t),
\]
and since $\beta_t(t)<2t$, one gets
\[
\Delta_{T_{t+1}-1}<t(t+1)=\Delta_{T_{t+1}}.
\]
Thus $(\Delta_n)$ is nondecreasing, with a strict increase at every transition to a new triangular interval.
\end{proof}

\begin{proposition}[Increment rule inside a triangular interval]\label{prop:increment-rule}
Let
\[
T_t\le n<n+1<T_{t+1},\qquad n=T_t+\nu,\qquad 0\le \nu<t.
\]
Then
\[
\Delta_{n+1}-\Delta_n\in\{0,1\}.
\]
More precisely,
\[
\Delta_{n+1}-\Delta_n=1
\]
if and only if $\nu+1$ is either a square or a pronic number, that is,
\[
\nu+1=q^2\quad\text{or}\quad \nu+1=q(q+1)
\]
for some integer $q\ge 1$. Otherwise,
\[
\Delta_{n+1}-\Delta_n=0.
\]
\end{proposition}

\begin{proof}
Since
\[
\Delta_{T_t+\nu}=t(t-1)+\beta_t(\nu),
\]
one has
\[
\Delta_{n+1}-\Delta_n=\beta_t(\nu+1)-\beta_t(\nu).
\]
By Proposition~\ref{prop:beta-formula}, the function $\beta_t(\nu)$ increases by one exactly when the budget crosses one of the threshold values
\[
1,2,4,6,9,12,\dots,
\]
namely at squares and pronic numbers.
\end{proof}

\begin{corollary}\label{cor:jump-thresholds}
On every triangular interval $[T_t,T_{t+1})$, the graph of $\Delta_n$ is a staircase function with upward jumps precisely at the values
\[
n=T_t+q^2\qquad\text{and}\qquad n=T_t+q(q+1),
\]
whenever these values lie in the interval.
\end{corollary}

\begin{proposition}[Square-root bounds for the surplus correction]\label{prop:surplus-correction-bounds}
Let
\[
T_t\le n<T_{t+1},\qquad n=T_t+\nu,\qquad 0\le \nu\le t.
\]
Then
\[
2\lfloor\sqrt{\nu}\rfloor-1\le \beta_t(\nu)\le 2\lfloor\sqrt{\nu}\rfloor,
\]
with the convention that the left-hand side is $0$ for $\nu=0$. Consequently,
\[
t(t-1)+2\lfloor\sqrt{\nu}\rfloor-1\le \Delta_n\le t(t-1)+2\lfloor\sqrt{\nu}\rfloor.
\]
\end{proposition}

\begin{proof}
If $\nu=0$, then $\beta_t(0)=0$ by definition, so the claim is immediate.

Assume now that $\nu\ge 1$, and set
\[
q:=\lfloor \sqrt{\nu}\rfloor.
\]
Then
\[
q^2\le \nu<(q+1)^2.
\]
If $q<t$, Proposition~\ref{prop:beta-formula} gives exactly two possibilities:
\[
\beta_t(\nu)=2q-1\quad\text{when}\quad q^2\le \nu<q(q+1),
\]
and
\[
\beta_t(\nu)=2q\quad\text{when}\quad q(q+1)\le \nu<(q+1)^2.
\]
Hence in all cases
\[
2q-1\le \beta_t(\nu)\le 2q.
\]
The only remaining case is $q=t$. Since $\nu\le t$ and $q^2\le \nu$, this forces
\[
t^2\le t,
\]
so necessarily $t=1$ and $\nu=1$. Then directly
\[
\beta_1(1)=1,
\]
which again satisfies
\[
2q-1\le \beta_t(\nu)\le 2q.
\]
Since $q=\lfloor\sqrt{\nu}\rfloor$, this proves the stated bound for $\beta_t(\nu)$. The corresponding estimate for $\Delta_n$ follows from Corollary~\ref{cor:exact-max-degree}.
\end{proof}

\begin{corollary}[Global asymptotic form of the maximal degree]\label{cor:global-asymptotic}
Let
\[
T_t\le n<T_{t+1},\qquad n=T_t+\nu,\qquad 0\le \nu\le t.
\]
Then
\[
2t\le 2n-\Delta_n\le 4t.
\]
Consequently,
\[
\Delta_n=2n-\Theta(\sqrt n)\qquad (n\to\infty).
\]
\end{corollary}

\begin{proof}
Using
\[
n=\frac{t(t+1)}{2}+\nu,\qquad \Delta_n=t(t-1)+\beta_t(\nu),
\]
we obtain
\[
2n-\Delta_n=t(t+1)+2\nu-t(t-1)-\beta_t(\nu)=2t+2\nu-\beta_t(\nu).
\]
Since $0\le \nu\le t$ and $0\le \beta_t(\nu)\le 2\nu$, it follows that
\[
2t\le 2n-\Delta_n\le 2t+2\nu\le 4t.
\]
Because $t\sim \sqrt{2n}$, the asymptotic form follows.
\end{proof}

\section{Computational profile of the degree landscape}
\label{sec:computational}

In this section we specify the computational scope of the paper and isolate the numerical outputs most directly aligned with the theorem part. The exact formula for the maximal degree has already been established in Section~\ref{sec:max-degree}, so the role of the present section is not to supply further proofs, but to record a compact descriptive profile of the degree landscape on a finite range.

\subsection{Computational setup}

The theoretical results of Section~\ref{sec:max-degree} reduce the extremal degree problem to a much smaller search space than the full vertex set of $G_n$, allowing a computational study guided by the structure of Section~\ref{sec:max-degree} rather than by explicit construction of the full adjacency relation.

For a partition
\[
\lambda=(L_1^{m_1},L_2^{m_2},\dots,L_r^{m_r})\vdash n,\qquad L_1>\cdots>L_r>0,
\]
the degree is computed directly from Proposition~\ref{prop:degree-formula}:
\[
\deg(\lambda)=r(r-1)+\sum_{i=1}^r \mathbf 1_{m_i>1}+\sum_{i=1}^r \mathbf 1_{g_i>1},\qquad g_i=L_i-L_{i+1},\quad L_{r+1}=0.
\]
Accordingly, the computational part of the paper does not require explicit construction of all edges of $G_n$. Instead, for each partition $\lambda\vdash n$, we record its support size, gap pattern, multiplicity pattern, conjugacy type, and degree.

This approach is especially efficient in the extremal regime. If
\[
n=T_t+\nu,\qquad 0\le \nu\le t,
\]
then Theorem~\ref{thm:max-support-principle} implies that every extremal vertex has support size $t=\rho(n)$. Thus extremal search may be restricted to the maximal-support stratum. However, this stratum is not parametrized by a single weak partition of $\nu$: one must track both gap surplus and multiplicity surplus. The correct description uses both gap excesses and multiplicity excesses.

\begin{lemma}[Encoding of the maximal-support stratum]\label{lem:max-support-encoding}
Let
\[
n=T_t+\nu,\qquad 0\le \nu\le t.
\]
A partition $\lambda\vdash n$ has support size $t$ if and only if it can be written uniquely in the form
\[
\lambda=(L_1^{1+\mu_1},L_2^{1+\mu_2},\dots,L_t^{1+\mu_t}),\qquad L_1>\cdots>L_t>0,
\]
where
\[
\alpha_j:=g_j-1\ge 0\qquad (1\le j\le t),\qquad \mu_i\ge 0\qquad (1\le i\le t),
\]
and
\[
L_i=t+1-i+\sum_{j=i}^t \alpha_j\qquad (1\le i\le t),
\]
with
\[
\nu=\sum_{j=1}^t j\,\alpha_j+\sum_{i=1}^t \mu_iL_i.
\]
Conversely, every such choice of nonnegative integers $\alpha_j,\mu_i$ determines a unique partition on the maximal-support stratum.
\end{lemma}

\begin{proof}
If $\lambda$ has support size $t$, define
\[
\alpha_j:=g_j-1=L_j-L_{j+1}-1\ge 0,\qquad \mu_i:=m_i-1\ge 0.
\]
Then the telescoping relation for the support gives
\[
L_i=(t+1-i)+\sum_{j=i}^t \alpha_j.
\]
Since $n=T_t+\nu$ and $\lambda$ has support size $t$, Lemma~\ref{lem:triangular-mass} yields
\[
\nu=\sum_{j=1}^t j\,\alpha_j+\sum_{i=1}^t \mu_iL_i.
\]
This representation is unique because the data $\alpha_j$ and $\mu_i$ are determined by the gap and multiplicity patterns of $\lambda$.

Conversely, given nonnegative integers $\alpha_j,\mu_i$ satisfying the displayed formulas, define $L_i$ as above and set
\[
\lambda=(L_1^{1+\mu_1},\dots,L_t^{1+\mu_t}).
\]
Then $L_1>\cdots>L_t>0$, so $\lambda$ has support size $t$, and Lemma~\ref{lem:triangular-mass} shows that its size is exactly $T_t+\nu=n$.
\end{proof}

Lemma~\ref{lem:max-support-encoding} is computationally important: instead of searching for extremal vertices among all partitions of $n$, one may search over weighted surplus data on the maximal-support stratum. In particular, the extremal computations may be organized using the nonnegative integer data
\[
(\alpha_1,\dots,\alpha_t;\mu_1,\dots,\mu_t)
\]
subject to the weighted budget equation
\[
\nu=\sum_{j=1}^t j\,\alpha_j+\sum_{i=1}^t \mu_iL_i.
\]
Since $\nu\le t$ in the extremal regime, this search space is modest compared with the full set of partitions of $n$.

In the computations reported below, we fix the range
\[
1\le n\le 60.
\]
All computations were performed in Python by exhaustive enumeration of partitions of $n$, with degrees evaluated directly from Proposition~\ref{prop:degree-formula}. For spectral data such as $\Spec_D(n)$, $s(n)=|\Spec_D(n)|$, and $H_n(d)=|D_d(n)|$, we enumerate all partitions of $n$ and evaluate the degree formula directly. For extremal data such as $\Delta_n$, $m_\Delta(n)=|\operatorname{MaxDeg}(n)|$, and $m_\Delta^{\mathrm{sc}}(n)=|\operatorname{MaxDeg}(n)\cap Ax_n|$, we exploit the exact formula for $\Delta_n$ from Corollary~\ref{cor:exact-max-degree} together with Lemma~\ref{lem:max-support-encoding}, and use the maximal-support encoding as a structural check on the extremal layer.

The numerical purpose of this section is therefore limited and explicit. We do not attempt a full asymptotic analysis of the degree spectrum or a complete classification of extremal partitions. Instead, we record a compact package of data that complements the theorem part of the paper: extremal multiplicities, selected extremal shapes, spectrum sizes, representative histograms, and a first comparison with the self-conjugate axis.

\subsection{Extremal multiplicity and representative extremal shapes}

The exact formula for $\Delta_n$ determines the maximal degree value for every $n$, but it does not determine the size or internal structure of the extremal set
\[
\operatorname{MaxDeg}(n)=\{\lambda\vdash n:\deg(\lambda)=\Delta_n\}.
\]
The first new numerical invariant beyond Section~\ref{sec:max-degree} is therefore the extremal multiplicity
\[
m_\Delta(n):=|\operatorname{MaxDeg}(n)|.
\]
We also record
\[
m_\Delta^{\mathrm{sc}}(n):=|\operatorname{MaxDeg}(n)\cap Ax_n|,
\]
the number of self-conjugate extremal vertices.

Table~\ref{tab:master-degree-data} records the basic degree data on the computed range. Its $\Delta_n$ column agrees throughout with the exact formula from Section~\ref{sec:max-degree} and therefore serves as a numerical consistency check rather than as an independent source of evidence. The genuinely new information lies in the columns for $m_\Delta(n)$, $m_\Delta^{\mathrm{sc}}(n)$, and $s(n)$.

On the computed range $1\le n\le 60$, the extremal multiplicity is at most $22$ and takes only five values:
\[
1,\quad 2,\quad 6,\quad 8,\quad 22.
\]
In particular, even on this modest range the extremal layer is already visibly nontrivial.
In $41$ of the first $60$ cases, the extremal set consists either of a single partition or of a single conjugate pair. The largest value on this range is
\[
m_\Delta(44)=22,\qquad m_\Delta(53)=22,
\]
and in each of these two cases the extremal layer contains two self-conjugate vertices.

\setlength{\LTleft}{0pt}
\setlength{\LTright}{0pt}
\small
\begin{longtable}{ccccccc}
\caption{Degree data for $1\le n\le 60$: maximal degree, extremal multiplicity, axial contact, and spectrum size.}
\label{tab:master-degree-data}\\
\toprule
$n$ & $\rho(n)$ & $\nu$ & $\Delta_n$ & $m_\Delta(n)$ & $m_\Delta^{\mathrm{sc}}(n)$ & $s(n)$\\
\midrule
\endfirsthead
\multicolumn{7}{l}{\tablename~\thetable\ (continued)}\\
\toprule
$n$ & $\rho(n)$ & $\nu$ & $\Delta_n$ & $m_\Delta(n)$ & $m_\Delta^{\mathrm{sc}}(n)$ & $s(n)$\\
\midrule
\endhead
\midrule
\multicolumn{7}{r}{\emph{continued on next page}}\\
\endfoot
\bottomrule
\endlastfoot
1 & 1 & 0 & 0 & 1 & 1 & 1 \\
2 & 1 & 1 & 1 & 2 & 0 & 1 \\
3 & 2 & 0 & 2 & 1 & 1 & 2 \\
4 & 2 & 1 & 3 & 2 & 0 & 3 \\
5 & 2 & 2 & 4 & 1 & 1 & 3 \\
6 & 3 & 0 & 6 & 1 & 1 & 5 \\
7 & 3 & 1 & 7 & 2 & 0 & 4 \\
8 & 3 & 2 & 8 & 1 & 1 & 7 \\
9 & 3 & 3 & 8 & 6 & 0 & 7 \\
10 & 4 & 0 & 12 & 1 & 1 & 9 \\
11 & 4 & 1 & 13 & 2 & 0 & 8 \\
12 & 4 & 2 & 14 & 1 & 1 & 11 \\
13 & 4 & 3 & 14 & 6 & 0 & 10 \\
14 & 4 & 4 & 15 & 2 & 0 & 13 \\
15 & 5 & 0 & 20 & 1 & 1 & 13 \\
16 & 5 & 1 & 21 & 2 & 0 & 15 \\
17 & 5 & 2 & 22 & 1 & 1 & 14 \\
18 & 5 & 3 & 22 & 6 & 0 & 17 \\
19 & 5 & 4 & 23 & 2 & 0 & 17 \\
20 & 5 & 5 & 23 & 8 & 0 & 19 \\
21 & 6 & 0 & 30 & 1 & 1 & 21 \\
22 & 6 & 1 & 31 & 2 & 0 & 21 \\
23 & 6 & 2 & 32 & 1 & 1 & 21 \\
24 & 6 & 3 & 32 & 6 & 0 & 24 \\
25 & 6 & 4 & 33 & 2 & 0 & 25 \\
26 & 6 & 5 & 33 & 8 & 0 & 26 \\
27 & 6 & 6 & 34 & 1 & 1 & 26 \\
28 & 7 & 0 & 42 & 1 & 1 & 29 \\
29 & 7 & 1 & 43 & 2 & 0 & 27 \\
30 & 7 & 2 & 44 & 1 & 1 & 31 \\
31 & 7 & 3 & 44 & 6 & 0 & 29 \\
32 & 7 & 4 & 45 & 2 & 0 & 33 \\
33 & 7 & 5 & 45 & 8 & 0 & 34 \\
34 & 7 & 6 & 46 & 1 & 1 & 36 \\
35 & 7 & 7 & 46 & 6 & 0 & 36 \\
36 & 8 & 0 & 56 & 1 & 1 & 37 \\
37 & 8 & 1 & 57 & 2 & 0 & 37 \\
38 & 8 & 2 & 58 & 1 & 1 & 40 \\
39 & 8 & 3 & 58 & 6 & 0 & 40 \\
40 & 8 & 4 & 59 & 2 & 0 & 43 \\
41 & 8 & 5 & 59 & 8 & 0 & 42 \\
42 & 8 & 6 & 60 & 1 & 1 & 45 \\
43 & 8 & 7 & 60 & 6 & 0 & 43 \\
44 & 8 & 8 & 60 & 22 & 2 & 47 \\
45 & 9 & 0 & 72 & 1 & 1 & 48 \\
46 & 9 & 1 & 73 & 2 & 0 & 49 \\
47 & 9 & 2 & 74 & 1 & 1 & 48 \\
48 & 9 & 3 & 74 & 6 & 0 & 51 \\
49 & 9 & 4 & 75 & 2 & 0 & 52 \\
50 & 9 & 5 & 75 & 8 & 0 & 54 \\
51 & 9 & 6 & 76 & 1 & 1 & 55 \\
52 & 9 & 7 & 76 & 6 & 0 & 57 \\
53 & 9 & 8 & 76 & 22 & 2 & 56 \\
54 & 9 & 9 & 77 & 2 & 0 & 58 \\
55 & 10 & 0 & 90 & 1 & 1 & 59 \\
56 & 10 & 1 & 91 & 2 & 0 & 61 \\
57 & 10 & 2 & 92 & 1 & 1 & 63 \\
58 & 10 & 3 & 92 & 6 & 0 & 63 \\
59 & 10 & 4 & 93 & 2 & 0 & 63 \\
60 & 10 & 5 & 93 & 8 & 0 & 65 \\
\end{longtable}
\normalsize
\clearpage

To make the top degree layer more concrete, Table~\ref{tab:extremal-shapes} lists selected representatives of extremal conjugacy orbits. The exact theorem-level cases $n=T_t$, $T_t+1$, and $T_t+2$ appear at $n=15,16,17$, while the later examples illustrate the first extremal layers with several conjugacy orbits. In particular, $n=20$, $n=35$, and $n=60$ already exhibit several distinct conjugacy orbits, whereas $n=27$ shows that unique self-conjugate extremizers persist beyond the first overlevels covered by Section~\ref{sec:max-degree}.

\begin{table}[H]
\centering
\small
\begin{tabular}{c p{0.52\textwidth} c c}
\toprule
$n$ & representative orbit member in $\operatorname{MaxDeg}(n)$ & type & orbit size\\
\midrule
15 & (5,4,3,2,1) & self-conjugate & 1 \\
16 & (6,4,3,2,1) & conjugate pair & 2 \\
17 & (6,4,3,2,1,1) & self-conjugate & 1 \\
20 & (8,5,3,2,1,1) & conjugate pair & 2 \\
20 & (7,5,4,2,1,1) & conjugate pair & 2 \\
20 & (7,5,3,2,2,1) & conjugate pair & 2 \\
20 & (7,5,3,2,1,1,1) & conjugate pair & 2 \\
27 & (8,6,4,3,2,2,1,1) & self-conjugate & 1 \\
35 & (10,8,6,4,3,2,1,1) & conjugate pair & 2 \\
35 & (10,7,5,4,3,2,2,1,1) & conjugate pair & 2 \\
35 & (9,7,6,4,3,2,2,1,1) & conjugate pair & 2 \\
60 & (13,10,8,7,6,5,4,3,2,1,1) & conjugate pair & 2 \\
60 & (12,10,9,7,6,5,4,3,2,1,1) & conjugate pair & 2 \\
60 & (12,10,8,7,6,5,4,3,2,2,1) & conjugate pair & 2 \\
60 & (12,10,8,7,6,5,4,3,2,1,1,1) & conjugate pair & 2 \\
\bottomrule
\end{tabular}
\caption{Selected representatives of extremal conjugacy orbits. For each displayed value of $n$, one representative from each conjugation orbit in $\operatorname{MaxDeg}(n)$ is listed; for non-self-conjugate orbits, we choose the lexicographically larger member.}
\label{tab:extremal-shapes}
\end{table}

\subsection{Degree spectra and selected histograms}

While the maximal degree $\Delta_n$ is now known exactly, the full degree spectrum
\[
\Spec_D(n)=\{\deg(\lambda):\lambda\vdash n\}
\]
contains substantially richer information about the global degree landscape of $G_n$. In particular, it records how many distinct local degree values occur and how the multiplicities of these values are distributed across the vertex set.

The spectrum size
\[
s(n):=|\Spec_D(n)|
\]
is already listed in Table~\ref{tab:master-degree-data}. On the range $1\le n\le 60$, it grows overall from $1$ to $65$, although not monotonically: local plateaus and small drops occur repeatedly. This suggests that the degree landscape becomes progressively more differentiated as $n$ grows, though not in a strictly uniform way.

The histograms
\[
H_n(d):=|D_d(n)|
\]
give a more detailed view. Figure~\ref{fig:degree-histograms} displays the degree histograms for $n=20$, $40$, and $60$. For $n=20$, $40$, and $60$, the computed histograms show several distinct peaks, suggesting a multimodal structure. This is consistent with the structural decomposition from Proposition~\ref{prop:degree-formula}: degree values are assembled from support contributions together with activated multiplicity and gap bonuses, so one expects several overlapping clusters rather than a single bulk profile. At the same time, the extreme upper tail remains very thin compared with the main interior peaks.

\begin{figure}[H]
\centering
\includegraphics[width=\textwidth]{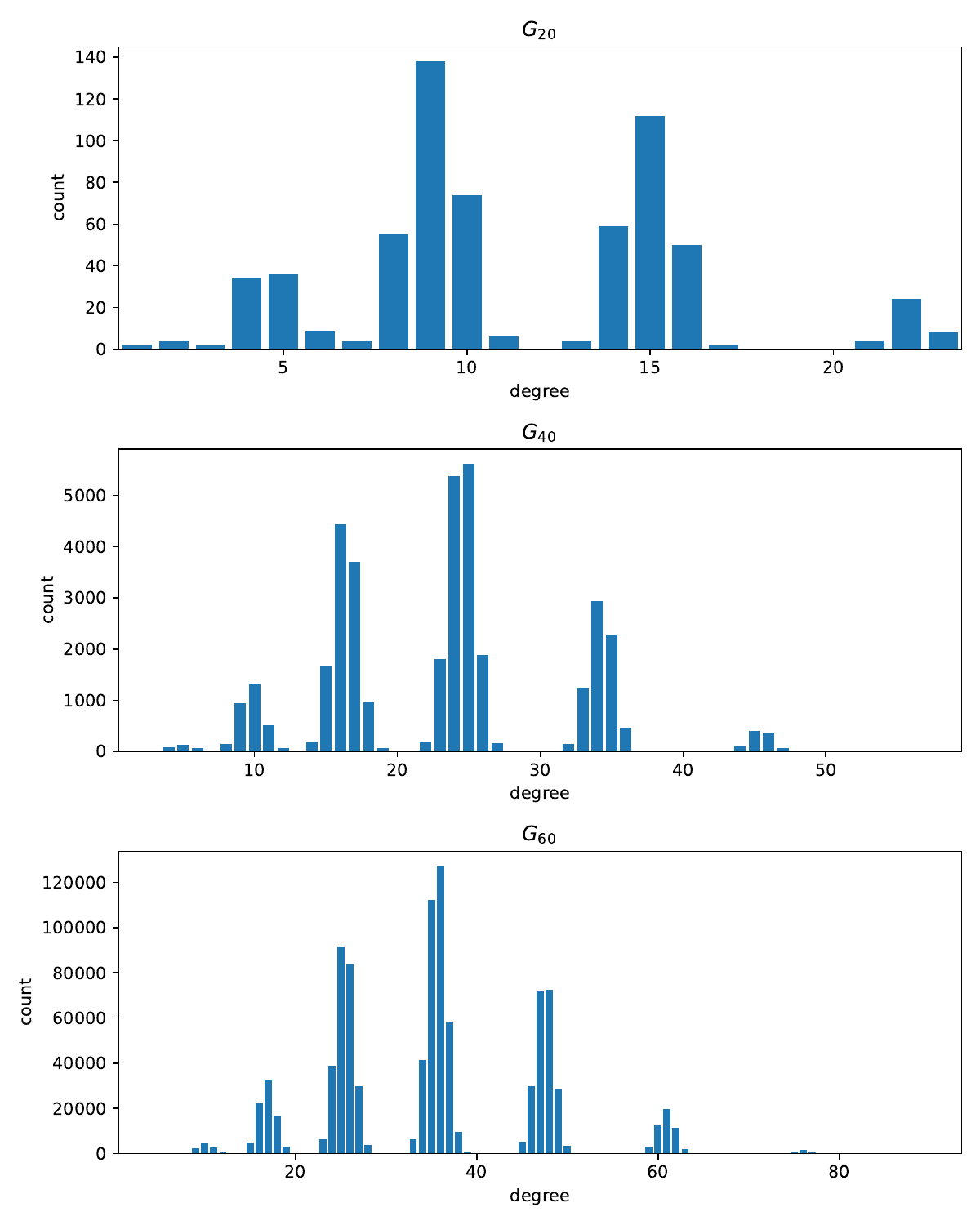}
\caption{Degree histograms for $G_{20}$, $G_{40}$, and $G_{60}$. The three computed histograms show several distinct peaks, reflecting the contribution of different support strata and bonus patterns to the degree decomposition.}
\label{fig:degree-histograms}
\end{figure}

\subsection{A first comparison with the self-conjugate axis}

The degree of a vertex is a local invariant, but a natural next question is whether the extremal layer interacts in a systematic way with the global geometry of the graph. Here we keep this comparison deliberately modest and restrict it to the most stable feature available at the current stage, namely contact with the self-conjugate axis.

The column $m_\Delta^{\mathrm{sc}}(n)$ in Table~\ref{tab:master-degree-data} gives the first axial comparison. On the range $1\le n\le 60$, the extremal layer meets the self-conjugate axis for $25$ values of $n$. This includes the exact theorem-level cases $n=T_t$ and $n=T_t+2$, but not only those. Additional self-conjugate extremizers occur, for example, at
\[
n=27,\qquad n=34,\qquad n=42,\qquad n=51,
\]
and at
\[
n=44\quad\text{and}\quad n=53,
\]
where the extremal set has size $22$ and contains two self-conjugate vertices.

We do not claim a general law on the basis of this finite range. The data show only that axial contact persists well beyond the first exact cases. This suggests that axial contact may be a persistent feature of the extremal landscape, warranting further study. More detailed statistics involving distances to the axis, the spine, the boundary framework, or the simplex stratification are natural next steps, but they are not required for this first version and are therefore left for later work.

\section{Conclusion and open problems}
\label{sec:conclusion}

We have introduced the degree landscape of the partition graph and placed it on a precise structural footing. The central result is that a local invariant --- the degree --- becomes globally tractable once the local degree formula is combined with a triangular decomposition of the total mass of a partition. This leads to a weighted budget principle for bonus activations and, ultimately, to a complete determination of the maximal degree.

The main theorem shows that if
\[
\rho(n)=\max\{r:T_r\le n\},\qquad \nu=n-T_{\rho(n)},
\]
then
\[
\Delta_n=\rho(n)\bigl(\rho(n)-1\bigr)+\beta_{\rho(n)}(\nu).
\]
Thus maximal degree is governed exactly by two ingredients: the largest possible support size and the cheapest available pattern of gap and multiplicity bonuses. In particular, every maximal-degree vertex lies on the maximal-support stratum, and the maximal-degree sequence is organized by a square--pronic threshold rule on each triangular interval.

We also obtained the first exact extremal classifications beyond the triangular levels themselves. At $n=T_t$, the unique extremal partition is the staircase partition. At $n=T_t+1$, the extremal set is a conjugate pair. At $n=T_t+2$, the unique extremal vertex is again self-conjugate. These first cases already show that the geometry of the top degree layer is subtler than the single value $\Delta_n$ alone.

At the same time, much of the global degree picture remains open. The exact value of $\Delta_n$ is now known, but the internal structure of the extremal set, the full degree spectrum, and the placement of high-degree vertices relative to the global morphology still require further analysis. The finite computation in Section~\ref{sec:computational} is intended only as a descriptive first step in that direction.

On the computed range $1\le n\le 60$, the numerical picture is already informative. The extremal multiplicity is at most $22$ and takes only the five values listed in Section~\ref{sec:computational}, the spectrum size reaches $65$, the displayed histograms show several distinct peaks, and contact with the self-conjugate axis persists beyond the first exact cases proved in Section~\ref{sec:max-degree}. These observations do not amount to new theorems, but they help identify the next structural questions and indicate where a fuller degree-theoretic morphology should begin.

We conclude with several natural problems.

\begin{problem}[Extremal multiplicity]
Determine the sequence
\[
m_\Delta(n)=|\operatorname{MaxDeg}(n)|.
\]
In particular, does $m_\Delta(T_t+\nu)$ depend primarily on the surplus $\nu$, or does it retain substantial dependence on the triangular level $t$?
\end{problem}

\begin{problem}[Classification of extremal partitions]
Classify the extremal set $\operatorname{MaxDeg}(n)$ for general $n$. Are all extremal partitions staircase-derived in a suitable sense, or do genuinely different extremal families eventually appear?
\end{problem}

\begin{problem}[Self-conjugate extremizers]
Determine when the extremal set contains self-conjugate partitions. More specifically, describe the sequence
\[
m_\Delta^{\mathrm{sc}}(n)=|\operatorname{MaxDeg}(n)\cap Ax_n|.
\]
\end{problem}

\begin{problem}[Degree spectrum size]
Determine the growth and structure of the sequence
\[
s(n)=|\Spec_D(n)|.
\]
How dense is the degree spectrum inside the interval
\[
[1,\Delta_n]?
\]
Do large internal gaps persist?
\end{problem}

\begin{problem}[Upper-tail thickness]
Study the near-extremal counts
\[
U_n(c)=|\{\lambda\vdash n:\deg(\lambda)\ge \Delta_n-c\}|.
\]
Does the upper tail of the degree spectrum exhibit stable qualitative behavior as $n$ grows?
\end{problem}

\begin{problem}[Geometric concentration]
Describe the placement of extremal and near-extremal vertices relative to the self-conjugate axis $Ax_n$, the spine $Sp_n$, the boundary framework, and the simplex layers. Do the highest degree layers concentrate near the central axial geometry of $G_n$ as $n$ grows?
\end{problem}

\begin{problem}[Interaction with simplex stratification]
Clarify the relation between degree and local simplex dimension. In particular, to what extent are large degree and large local simplex dimension correlated, and where do these two landscapes diverge?
\end{problem}

\begin{problem}[Asymptotic spectral profile]
Go beyond the maximal degree and determine asymptotic information about the full histogram
\[
H_n(d)=|D_d(n)|
\]
or suitable normalized forms of it.
\end{problem}

The degree-theoretic viewpoint introduced here should be seen as a bridge between local transfer data and the already established large-scale geometry of the partition graph. It isolates one of the simplest local invariants, shows that its extremal behavior is governed by a rigid global mechanism, and opens the way to a more systematic numerical study of the internal degree structure of $G_n$.

\section*{Acknowledgements}

The author acknowledges the use of ChatGPT (OpenAI) for discussion, structural planning, and editorial assistance during the preparation of this manuscript. All mathematical statements, proofs, computations, and final wording were checked and approved by the author, who takes full responsibility for the contents of the paper.

\end{document}